\documentclass[11pt]{article}%
\usepackage{amssymb,amsmath,amsfonts,amsthm,array,bm,bbm,color}%
\usepackage{graphicx}
\providecommand{\U}[1]{\protect\rule{.1in}{.1in}}

\numberwithin{equation}{section}

\newtheorem{theorem}{Theorem}[section]
\newtheorem{lemma}[theorem]{Lemma}

\newtheorem{remark}[theorem]{Remark}

\def\<{\langle}
\def\>{\rangle}

\def\R{\mathbb{R}}
\def\T{\mathbb{T}}
\def\Z{\mathbb{Z}}

\begin{document}

\title{Uniqueness of weak solutions to the limit resonant equation of 3D rotating Navier-Stokes equations}

\author{Dejun Luo\footnote{Email: luodj@amss.ac.cn. Key Laboratory of RCSDS, Academy of Mathematics and Systems Science, Chinese Academy of Sciences, Beijing 100190, China and School of Mathematical Sciences, University of Chinese Academy of Sciences, Beijing 100049, China} }
\maketitle

\begin{abstract}
The limit resonant equation of the 3D rotating Navier-Stokes equations is obtained by taking large rotation limit. This equation has a nonlinear term with restricted interactions between Fourier modes, and thus it enjoys better regularity estimates than those of the classical 3D Navier-Stokes equations. Such estimates enable us to prove uniqueness of weak solutions to the limit resonant equation.
\end{abstract}



\section{Introduction}

Consider the 3D rotating Navier-Stokes equations (or the 3D Navier-Stokes equations with Coriolis force) on the torus $\T^3$:
  \begin{equation}\label{3D-RNSE}
  \left\{ \aligned
  & \partial_t U + U\cdot\nabla U + \Omega e_3\times U + \nabla p= \nu \Delta U , \\
  & \nabla\cdot U= 0, \quad  U(0,\cdot) = U_0,
  \endaligned \right.
  \end{equation}
where $U=(U_1, U_2, U_3)^\ast$ is the fluid velocity field, $p$ the scalar pressure field, $\nu>0$ the viscosity coefficient and $\Omega>0$ the angular velocity around the vertical axis $e_3=(0,0,1)^\ast$. Let $H^s:= H^s(\T^3, \R^3) \, (s\in \R)$ be the usual Sobolev space of divergence free vector fields on $\T^3$; we shall write $H=L^2$ for $H^0$. In the important work \cite{BMN99}, Babin, Mahalov and Nicolaenko have shown that, given $\alpha>1/2$ and divergence free vector field $U_0$ in a ball of $ H^\alpha$, if $\Omega$ is big enough, then the above system admits a global regular solution for all $t\ge 0$ and $\|U(t)\|_{H^\alpha}$ is bounded in $t$; cf. \cite[Theorem 1.1]{BMN99} and also \cite{BMN97, BMN01} for related studies.

We briefly recall the strategy of their proof and refer to the next section for precise definitions and notation. Let $\{E(-\Omega t) \}_{t\in \R}$ be the Poincar\'e propagator which is a group of unitary operators; introducing the van der Pol transformation
  $$U(t) = E(-\Omega t) u(t) ,$$
we obtain an equation in $u$ variables
  \begin{equation}\label{3D-RNSE-transformed}
  \left\{ \aligned
  & \partial_t u + B(\Omega t, u,u) =  \nu \Delta u, \\
  & B(\Omega t, u,u)= E(\Omega t)\, B(E(-\Omega t)\, u, E(-\Omega t)\, u).
  \endaligned \right.
  \end{equation}
The nonlinear term admits the decomposition
  $$B(\Omega t, u,u)= \tilde B( u,u) + B^{osc}(\Omega t, u,u),$$
where $\tilde B( u,u)$ is $\Omega t$-independent while the second term $B^{osc}(\Omega t, u,u)$ contains $\Omega t$-dependent terms. In the large rotation limit $\Omega\to \infty$, one can show that $B^{osc}(\Omega t, u,u)$ vanishes in a suitable sense and thus we arrive at the limit equation (called the limit resonant equation in \cite{BMN99} or the fast singular oscillating limit equation in \cite{BMN01}):
  \begin{equation}\label{limit-resonant-eq}
  \partial_t u + \tilde B( u,u) =  \nu \Delta u.
  \end{equation}
The nonlinear term $\tilde B( u,u)$ fulfills the usual properties of the nonlinearity of the classical 3D Navier-Stokes equations; moreover, due to limited interactions between Fourier modes, one can prove better regularity estimate on $\tilde B( u,u)$ (see \cite[Theorem 3.1]{BMN99}), which allows us to show global well-posedness of \eqref{limit-resonant-eq} for smooth initial data. Finally, for $\Omega$ large but finite, one can bootstrap the global regularity of solution to \eqref{3D-RNSE} from that of \eqref{limit-resonant-eq}.

In this paper, we are concerned with the well-posedness of \eqref{limit-resonant-eq} with $L^2$-initial data. The existence of weak solutions to \eqref{limit-resonant-eq} satisfying energy inequality follows from classical arguments; our purpose is to show that such solutions are also unique. We write $a\lesssim b$ if $a\le C b$ for some unimportant constant $C>0$.

\begin{theorem}\label{thm-key-estimate}
The following estimate holds:
  \begin{equation}\label{thm-key-estimate.1}
  \big|\big\<\tilde B(u,v), u\big\> \big| \lesssim \|u \|_{L^2} \|u\|_{H^1} \|v\|_{H^1} ,\quad \forall\, u ,\, v\in H^1.
  \end{equation}
As a result, Leray's weak solutions to \eqref{limit-resonant-eq} are unique.
\end{theorem}

The above estimate has the same form as that of the nonlinearity of 2D Navier-Stokes equations; it is the key ingredient for proving uniqueness of weak solutions to \eqref{limit-resonant-eq}. We shall prove \eqref{thm-key-estimate.1} by adapting the arguments in the proof of \cite[Theorem 3.1]{BMN99}. Unfortunately, in this case, we are unable to deduce similar assertions for the original 3D rotating Navier-Stokes equations \eqref{3D-RNSE}. We point out that, using the estimate of this paper, one can also derive uniqueness of weak solutions to \eqref{limit-resonant-eq} with suitable random noise.

Let us recall that, for 3D Navier-Stokes equations without Coriolis force (i.e. $\Omega=0$ in \eqref{3D-RNSE}), Leray proved in the fundamental work \cite{Leray33} global existence of weak solutions satisfying energy inequality, but their uniqueness remains open; on the other hand, for smooth initial data, strong solutions exist uniquely but their global regularity is unknown. Since then there have been numerous efforts but the overall solution theory remains roughly unchanged. As discussed at the beginning of \cite{BMN01}, the above rotating 3D Navier-Stokes equations \eqref{3D-RNSE} are equivalent to the 3D Navier-Stokes system without rotation, but with an initial data characterized by uniformly large vorticity; therefore, the results of Babin, Mahalov and Nicolaenko identify a class of initial data for which the Navier-Stokes system is globally well posed.

Though the uniqueness result in Theorem \ref{thm-key-estimate} is concerned with \eqref{limit-resonant-eq}, rather than the original 3D Navier-Stokes equations, it might be useful to recall various non-uniqueness results proved in recent years. Using convex integration method, non-unique weak solutions to 3D Euler equations were constructed in \cite{DLS09, DLS13, DLS14, BTDLIS, Isett}, which finally leads to the solution of Onsager's conjecture. This method was adapted by Buckmaster and Vicol \cite{BV19} to construct non-unique weak solutions to the 3D Navier-Stokes equations; see \cite{HZZ21, HZZ22, HLP23, Pap23} for extensions to the stochastic Euler/Navier-Stokes equations with various random perturbations. However, the weak solutions constructed from convex integration method do not fulfill the energy inequality of Leray solutions. In the recent work \cite{ABC22}, Albritton, Bru\'e and Colombo were able to prove non-uniqueness of weak solutions, in the Leray class, to 3D Navier-Stokes equations, but they had to construct at the same time an extra forcing term. Their method is based on adapting the constructions in \cite{Vis18a, Vis18b, ABCDGJK} which show the non-uniqueness of 2D Euler equations with $L^p$-initial vorticity for $p<\infty$. We finish this paragraph by mentioning a few results in the theory of regularization by noise, see e.g. \cite{GHV14, RZZ14, FL21, FHLN, Luo23} and the references therein.

This short paper is organized as follows. In Section 2 we introduce some notation and give more information on the limit resonant equation \eqref{limit-resonant-eq}. Then we prove in Section 3 the main result.

\section{Notation and main result}

We follow the presentation in \cite[Sections 2 and 3]{FM12}, see also \cite[Section 2]{BMN99}. The vector product $e_3\times U$ can be written as
  $$ e_3\times U = \begin{pmatrix}
  0 &\ -1 &\ 0 \\
  1 &\ 0 &\ 0 \\
  0 &\ 0 &\ 0
  \end{pmatrix} U . $$
Let $J$ be the skew-symmetric matrix. Consider the linear Poincar\'e problem, which is the linearized version of \eqref{3D-RNSE} with $\nu=0$:
  $$\partial_t \Phi + \Omega J\Phi = -\nabla p, \quad \nabla\cdot \Phi =0, $$
or equivalently,
  $$\partial_t \Phi + \Omega PJP \Phi =0,$$
where $P$ is the Leray projection on the space of divergence free vector fields. Let $E(-\Omega t) \Phi(0) = \Phi(t)$ be the Poincar\'e propagator, then $\{E(-\Omega t) \}_{t\in\R}$ is a one-parameter group consisting of unitary operators which preserve $H^s$-norms for all $s\in \R$. Applying $P$ to the momentum equation in \eqref{3D-RNSE} yields
  \begin{equation}\label{3D-RNSE-projected}
  \partial_t U + B( U, U) + \Omega PJP U = \nu \Delta U,
  \end{equation}
where
  $$B( U, U)= P(U\cdot \nabla U) = P((\nabla\times U) \times U) .$$
We introduce the van der Pol transformation
  $$U(t)= E(-\Omega t)\, u(t), $$
and we have $U(0) =u(0)$ since $E(0) =Id$. Writing equations \eqref{3D-RNSE-projected} in $u$ variables yields the system \eqref{3D-RNSE-transformed}.

Let $\Z^3_0= \Z^3\setminus \{0\}$ be nonzero lattice points, $\{\varphi_{k,\sigma} e^{2\pi i k\cdot x}\}_{k\in \Z^3_0,\sigma= \pm 1}$ be a CONS of $H= L^2$, the space consisting of divergence free vector fields with zero average; here $\varphi_{k,\sigma}\in C^3$ with $|\varphi_{k,\sigma}|=1$. The system can be chosen in such a way that $\varphi_{k,\sigma} e^{2\pi i k\cdot x}$ are eigenvectors of the curl operator, corresponding to eigenvalues $\pm |k|,\, k\in \Z^3_0$; that is,
  $$\nabla\times (\varphi_{k,\sigma} e^{2\pi i k\cdot x})= \sigma|k| \varphi_{k,\sigma} e^{2\pi i k\cdot x}.$$
We can now give the expression of $B(\Omega t, u,v)$ in \eqref{3D-RNSE-transformed}; assume that
  $$u= \sum_{k\in \Z^3_0} \sum_{\sigma= \pm 1} u_{k,\sigma} \varphi_{k,\sigma} e^{2\pi i k\cdot x}, \quad  v= \sum_{k\in \Z^3_0} \sum_{\sigma= \pm 1} v_{k,\sigma} \varphi_{k,\sigma} e^{2\pi i k\cdot x}, $$
then we have (cf. \cite[Section 3, page 202]{FM12})
  $$B(\Omega t, u,v) = -2\pi \sum_{\stackrel{k,m\in \Z^3_0}{\sigma_1, \sigma_2 = \pm 1}} \sigma_2 |m| e^{i\Omega t D(k,m,\sigma_1, \sigma_2)} u_{k,\sigma_1} v_{m,\sigma_2} P_{k+m} (\varphi_{k,\sigma_1} \times \varphi_{m,\sigma_2})\, e^{2\pi i (k+m)\cdot x}, $$
where $P_{k+m}$ is the projection operator onto the plane orthogonal to $k+m$, and
  $$ D(k,m,\sigma_1, \sigma_2) = -\sigma_1 \frac{k_3}{|k|} -\sigma_2 \frac{m_3}{|m|} + \sigma_{k,m,\sigma_1, \sigma_2} \frac{k_3 +m_3}{|k+m|} $$
with $\sigma_{k,m,\sigma_1, \sigma_2} \in \{1,-1\}$. Denote
  \begin{equation}\label{eq:Lambda-set}
  \Lambda = \big\{(k,m,\sigma_1, \sigma_2)\in \Z^3_0 \times \Z^3_0 \times \{+1, -1\}^2: D(k,m,\sigma_1, \sigma_2) =0 \big\}; \end{equation}
then we obtain the resonant operator
  $$\aligned
  \tilde B(u,v) &= -2\pi \sum_{(k,m,\sigma_1, \sigma_2)\in \Lambda} \sigma_2 |m| u_{k,\sigma_1} v_{m,\sigma_2} P_{k+m} (\varphi_{k,\sigma_1} \times \varphi_{m,\sigma_2})\, e^{2\pi i (k+m)\cdot x} \\
  &= -2\pi \sum_{n, \sigma} e^{2\pi i n\cdot x}  \sum_{(k,m,\sigma_1, \sigma_2)\in \Lambda, k+m=n} \sigma_2 |m| u_{k,\sigma_1} v_{m,\sigma_2} P_{n} (\varphi_{k,\sigma_1} \times \varphi_{m,\sigma_2}) \cdot \overline{\varphi_{n,\sigma}} ,
  \endaligned $$
that is,
  \begin{equation}\label{eq:B-tilde-components}
  \tilde B(u,v)_{n,\sigma} = -2\pi \sum_{(k,m,\sigma_1, \sigma_2)\in \Lambda, k+m=n} \sigma_2 |m| u_{k,\sigma_1} v_{m,\sigma_2} P_{n} (\varphi_{k,\sigma_1} \times \varphi_{m,\sigma_2}) \cdot \overline{\varphi_{n,\sigma}} .
  \end{equation}
On the other hand, the oscillatory operator is given by
  $$B^{\rm osc}(\Omega t, u,v) = -2\pi\!\! \sum_{(k,m,\sigma_1, \sigma_2)\notin \Lambda} \!\! \sigma_2 |m| e^{i\Omega t D(k,m,\sigma_1, \sigma_2)} u_{k,\sigma_1} v_{m,\sigma_2} P_{k+m} (\varphi_{k,\sigma_1} \times \varphi_{m,\sigma_2})\, e^{2\pi i (k+m)\cdot x} . $$

We present some properties of $\tilde B(u,v)$, see \cite[Lemma 2.1]{BMN99}. Recall that $H^s$ is the usual Sobolev space of divergence free vector fields on $\T^3$ with zero mean, let $\<\cdot, \cdot\>$ be the inner product in $L^2$ or the duality between $H^s$ and $H^{-s}$.

\begin{lemma}
Let $(u,v,w)\in H^{3/4} \times H^{3/4} \times H^1$, then
  $$\big\<\tilde B(u,v), w\big\> = \lim_{\Omega\to \infty} \frac1{2\pi} \int_0^{2\pi} \big\< B(\Omega s, u,v), w \big\>\, ds. $$
\end{lemma}

It is easy to see that the usual properties of the nonlinearity of 3D Navier-Stokes equations hold also for $\tilde B(u,v)$. For instance, from this lemma and the classical property of $B$, one has
  \begin{equation}\label{B-tilde-property}
  \big\<\tilde B(u,v), v\big\> = 0 ,\quad u,\, v\in H^1,
  \end{equation}
which implies, by polarization,
  $$\big\<\tilde B(u,v), w\big\> = - \big\<\tilde B(u,w), v\big\>,\quad u,\, v,\, w \in H^1. $$
In fact, better regularity estimates on $\tilde B$ can be established, see e.g. \eqref{thm-key-estimate.1} above and \cite[Theorem 3.1]{BMN99}. For any $u, v\in H^1$, combining estimate \eqref{thm-key-estimate.1} with the above identity, one has
  $$\big|\big\<\tilde B(u,u), v\big\> \big| = \big|\big\<\tilde B(u,v), u\big\> \big| \lesssim \|u \|_{L^2} \|u\|_{H^1} \|v\|_{H^1}, $$
and thus
  \begin{equation}\label{eq:B-tilde}
  \big\| \tilde B(u,u) \big\|_{H^{-1}} \lesssim \|u \|_{L^2} \|u\|_{H^1}.
  \end{equation}

Given $L^2$-initial data $u_0$, \eqref{limit-resonant-eq} admits the existence of weak solutions fulfilling
  \begin{equation}\label{eq:energy-ineq}
  \|u(t) \|_{L^2}^2 + 2\nu \int_0^t \|\nabla u(s) \|_{L^2}^2\, ds \le \|u_0 \|_{L^2}^2, \quad t\ge 0.
  \end{equation}
Using \eqref{eq:B-tilde} above, one can in fact prove that energy equality holds.

We finish this section with the following remark.

\begin{remark}\label{sec-2-rem}
Given a weak solution $u\in L^\infty(0,T; L^2) \cap L^2(0,T; H^1)$ to \eqref{limit-resonant-eq}, we have, by \eqref{eq:B-tilde},
  $$\int_0^T \big\| \tilde B(u(t), u(t)) \big\|_{H^{-1}}^2\, dt \lesssim \|u \|_{L^\infty_t L^2_x}^2 \int_0^T \|u(t)\|_{H^1}^2\, dt <+\infty. $$
This implies $\partial_t u \in L^2(0,T; H^{-1})$ and thus by Lions-Magenes theorem (see e.g. \cite{LM68}), $u\in C([0,T],L^2)$.
\end{remark}

\section{Proof of Theorem \ref{thm-key-estimate}}

We first present an estimate on restricted convolution which is motivated by \cite[Lemma 3.1, page 1148]{BMN99}.

\begin{lemma}\label{lem-restricted-convol}
Let $\chi(k,m,n)$ be the indicator function of some set $\Gamma \subset (\Z^3_0)^3$ such that it is symmetric: $\chi(k,m,n) = \chi(m,k,n) = \chi(k,n,m)$. Let $\alpha\ge 0$, $\beta$ fixed and
  \begin{equation}\label{eq:dyadic-estimate}
  \sup_{n\in \Z^3_0} \sum_{k: k+m+n=0, k\in S_i} \chi(k,m,n) |k|^{-\alpha} \le C_0\, 2^{i\beta}, \quad i\ge 0,
  \end{equation}
where $S_i$ is the dyadic block:
  $$ S_i = \big\{ k\in \Z^3_0: 2^i \le |k| < 2^{i+1} \big\}. $$
Then, for any sequences $\{u_n \}_{n\in \Z^3}, \{v_n \}_{n\in \Z^3}$ with $u_{(0,0,0)} = v_{(0,0,0)}= 0$, it holds
  $$ \sum_{k+m+n=0} |u_k|\, |m|\, |v_m|\, |u_n| \chi(k,m,n) \lesssim \big(\|u\|_{H^{\alpha/2}} \|u\|_{H^{\beta/2}} + \|u\|_{L^2} \|u\|_{H^{(\alpha+\beta)/2}} \big) \|v\|_{H^1}. $$
\end{lemma}

We remark that the sum on the left-hand side is different from that in \cite[Lemma 3.1]{BMN99} because the summands $|u_k|,\, |m|\, |v_m|$ are not symmetric in $k,m$. We postpone the proof of Lemma \ref{lem-restricted-convol} to the end of this section.

In order to apply Lemma \ref{lem-restricted-convol}, we prove the following crucial estimate which shows that \eqref{eq:dyadic-estimate} holds with $\alpha = \beta =1$ if $\chi(k,m,n)$ is the indicator function of
  \begin{equation} \label{eq:Lambda-set-1}
  \Gamma = \bigg\{(k,m,n)\in (\Z^3_0)^3: \pm \frac{k_3}{|k|} \pm \frac{m_3}{|m|} \pm \frac{n_3}{|n|} =0,\, k+m+n=0 \bigg\}. \end{equation}

\begin{lemma}\label{lem-dyadic-set}
Let $\chi(k,m,n)$ be the indicator function of $\Gamma$ defined above. Then one has
  $$ \sup_{n\in \Z^3_0} \sum_{k\in S_i} \chi(k,m,n) |k|^{-1} \le C_0\, 2^i, \quad i\ge 0.$$
\end{lemma}

\begin{proof}
Note that $\chi(k,m,n)\neq 0$ implies $k+m+n=0$. We distinguish three cases:

(i) If two vectors among $k,m,n$ have zero third components, e.g. $k_3= m_3=0$, then necessarily $n_3=0$. In this case, $\Gamma$ reduces to a set of 2D lattice points, and it is easy to show the estimate. In fact, in 2D, for any fixed $n\in \Z^2_0$,
  $$\sum_{k\in S_i} \chi(k,m,n) |k|^{-1} \le \sum_{k\in S_i} |k|^{-1} \lesssim \int_{2^i \le |x| \le 2^{i+1}} \frac{dx}{|x|} \lesssim 2^{i+1}- 2^{i} = 2^i.  $$

(ii) $k_3 m_3 n_3\ne 0$. In this case, the desired estimate was proved in \cite{BMN99}, page 1151.

(iii) One of the vectors has zero third component, while the other two have nonzero third components. For instance, $m_3=0$ but $k_3 n_3\ne 0$; then $(k,m,n)\in \Gamma$ implies $k_3+n_3=-m_3 =0$ and
  $$\pm \frac{k_3}{|k|} \pm \frac{n_3}{|n|}=0 \quad \Longrightarrow \quad |n|^2 = |k|^2.$$
This equation corresponds to a polynomial in $k_3$ of order at most two; given $k_1, k_2\in \Z$ and $n\in \Z^3_0$, it has at most two nonzero solutions, thus
  $$\aligned
  \sum_{k\in S_i} \chi(k,-k-n,n) |k|^{-1} &= \sum_{2^i \le |k| \le 2^{i+1}} \chi(k,-k-n,n) \big(k_1^2 +k_2^2 +k_3^2\big)^{-1/2} \\
  &\le 2+ 2\sum_{0 < |(k_1, k_2)| \le 2^{i+1}}\big(k_1^2 +k_2^2 \big)^{-1/2} \\
  &\le C_0\, 2^i,
  \endaligned $$
where $C_0$ is some absolute constant.
\end{proof}

Now we can provide the proof of Theorem \ref{thm-key-estimate}.

\begin{proof}[Proof of Theorem \ref{thm-key-estimate}]
We first prove the estimate \eqref{thm-key-estimate.1}. By \eqref{eq:B-tilde-components}, we have
  $$\big\<\tilde B(u,v), u\big\> = -2\pi \sum_{n,\sigma} \overline{u_{n,\sigma}} \sum_{(k,m,\sigma_1, \sigma_2)\in \Lambda, k+m=n} \sigma_2 |m|\, u_{k,\sigma_1} v_{m,\sigma_2} P_{n} (\varphi_{k,\sigma_1} \times \varphi_{m,\sigma_2}) \cdot \overline{\varphi_{n,\sigma}}
  $$
which implies, since $|\varphi_{n,\sigma}|=1$,
  $$\aligned
  \big|\big\<\tilde B(u,v), u\big\> \big| &\lesssim \sum_{n,\sigma} |u_{n,\sigma}| \sum_{(k,m,\sigma_1, \sigma_2)\in \Lambda, k+m=n}|m|\, |u_{k,\sigma_1}|\, |v_{m,\sigma_2}|.
  \endaligned $$
Note that $\sigma, \sigma_1, \sigma_2$ take values in $\{1,-1\}$. Let $|\tilde u_n|=\max\{|u_{n,1}|, |u_{n,-1}|\}$ for any $n\in \Z^3_0$, similarly for $|\tilde v_m|$, and $\chi(k,m,n)$ be the indicator function of $\Gamma$ in \eqref{eq:Lambda-set-1}. It is not difficult to see that
  $$\big|\big\<\tilde B(u,v), u\big\> \big| \lesssim \sum_{(k,m,n)\in (\Z^3_0)^3} \chi(k,m,n)\, |\tilde u_{n}|\, |\tilde u_{k}|\, |m|\, |\tilde v_{m}| .$$
Lemma \ref{lem-dyadic-set} shows that \eqref{eq:dyadic-estimate} holds with $\alpha =\beta =1$; therefore, by Lemma
\ref{lem-restricted-convol} and interpolation, we have
  $$\aligned
  \big|\big\<\tilde B(u,v), u\big\> \big| &\lesssim \big(\|\tilde u\|_{H^{1/2}}^2 + \|\tilde u\|_{L^2} \|\tilde u\|_{H^{1}} \big) \|\tilde v\|_{H^1} \\
  &\lesssim \|\tilde u\|_{L^2} \|\tilde u\|_{H^1} \|\tilde v\|_{H^1}\\
  &\lesssim \|u\|_{L^2} \|u\|_{H^1} \|v\|_{H^1}.
  \endaligned $$
Thus we obtain \eqref{thm-key-estimate.1}.

With estimate \eqref{thm-key-estimate.1} at hand, one can prove the uniqueness of weak solutions by following the classical arguments. We provide the proof for completeness. Indeed, let $u$ and $v$ be two weak solutions to \eqref{limit-resonant-eq} satisfying the energy inequality, then we have
  $$\partial_t (u-v) + \tilde B(u,u) - \tilde B(v, v) = \nu \Delta (u-v)$$
which holds as an equality in $H^{-1}$. Thanks to Remark \ref{sec-2-rem}, we know that $\partial_t (u-v) \in L^2(0,T; H^{-1})$; therefore, by Lions-Magenes theorem, we have
  $$\frac12 \frac{d}{d t}\|u-v \|_{L^2}^2 = \<u-v, \partial_t (u-v)\>= -\big\<u-v, \tilde B(u,u) - \tilde B(v, v)\big\> + \nu \<u-v, \Delta (u-v)\>. $$
We have $\tilde B(u,u) - \tilde B(v, v) = \tilde B(u-v, u) + \tilde B(v, u-v)$; by \eqref{B-tilde-property} and estimate \eqref{thm-key-estimate.1}, we obtain
  $$\aligned
  \frac12 \frac{d}{d t}\|u-v \|_{L^2}^2 + \nu \|\nabla(u-v)\|_{L^2}^2 &= -\big\<u-v, \tilde B(u-v, u)\big\> \\
  &\le C \|u-v\|_{L^2} \|u-v\|_{H^1} \|u\|_{H^1} \\
  &\le C' \|u-v\|_{L^2} \|\nabla(u-v)\|_{L^2} \|\nabla u\|_{L^2}.
  \endaligned $$
Cauchy's inequality implies
  $$C' \|u-v\|_{L^2} \|\nabla(u-v)\|_{L^2} \|\nabla u\|_{L^2} \le \nu\|\nabla(u-v)\|_{L^2}^2+ C_\nu \|u-v\|_{L^2}^2 \|\nabla u\|_{L^2}^2, $$
and thus
  $$\aligned
  \frac12 \frac{d}{d t}\|u-v \|_{L^2}^2 &\le C_\nu \|u-v\|_{L^2}^2 \|\nabla u\|_{L^2}^2.
  \endaligned $$
As $t\to \|\nabla u(t)\|_{L^2}^2$ is integrable and $t\in \|u(t)-v(t) \|_{L^2}^2$ is continuous in $t\in [0,T]$, Gronwall's inequality gives us that $\|u(t)-v(t) \|_{L^2} \equiv 0$ for all $t\in [0,T]$. This finishes the proof of uniqueness.
\end{proof}

Finally we prove Lemma \ref{lem-restricted-convol}, following the idea of \cite[Lemma 3.1]{BMN99}, see pages 1148--1149.

\begin{proof}[Proof of Lemma \ref{lem-restricted-convol}]
We rewrite the sum into six terms, according to the order of the norms $|k|, |m|,|n|$:
  $$\sum_{k+m+n=0} \chi(k,m,n)\, |u_{n}|\, |u_{k}|\, |m|\, |v_{m}|\le \sum_{i=1}^6 J_i, $$
where, setting $\tilde \chi(k,m,n)= \chi(k,m,n)\, {\bf 1}_{\{k+m+n=0\}}$,
  $$\aligned
  J_1 &= \sum_{|n|\ge |k|\ge |m|} \tilde \chi(k,m,n) |u_n|\, |u_k|\, |m|\, |v_m|,\\
  J_2 &= \sum_{|n|\ge |m|\ge |k|} \tilde \chi(k,m,n) |u_n|\, |u_k|\, |m|\, |v_m|,\\
  J_3 &= \sum_{|k|\ge |m|\ge |n|} \tilde \chi(k,m,n) |u_n|\, |u_k|\, |m|\, |v_m|,\\
  J_4 &= \sum_{|k|\ge |n|\ge |m|} \tilde \chi(k,m,n) |u_n|\, |u_k|\, |m|\, |v_m|,\\
  J_5 &= \sum_{|m|\ge |k|\ge |n|} \tilde \chi(k,m,n) |u_n|\, |u_k|\, |m|\, |v_m|,\\
  J_6 &= \sum_{|m|\ge |n|\ge |k|} \tilde \chi(k,m,n) |u_n|\, |u_k|\, |m|\, |v_m|.\\
  \endaligned $$
As the roles of $n$ and $k$ are symmetric, we see that
  $$J_1 = J_4, \quad J_2= J_3, \quad J_5= J_6. $$
So it suffices to estimate $J_1, J_2$ and $J_5$.

\emph{Step 1: estimate of $J_1$}. Note that $\tilde\chi(k,m,n)\neq 0$ implies $k+m+n=0$. We have
  $$\aligned
  J_1 &= \sum_k \sum_{n, |n|\ge |k|\ge |m|} \tilde\chi(k,m,n) |u_n|\, |u_k|\, |m|\, |v_{m}| \\
  &= \sum_{i\ge 0} \sum_{k\in S_i} \sum_{n, |n|\ge |k|\ge |m|} \tilde\chi(k,m,n) |u_n|\, |u_k|\, |m|\, |v_{m}| \\
  &\le \sum_{i\ge 0} \sum_{k\in S_i} \sum_{n\in S_i\cup S_{i+1}} \chi(k,-k-n,n) |u_n|\, |u_k|\, |k+n|\, |v_{-k-n}| ,
  \endaligned $$
where the last step is due to the fact that $2|k|\ge |n| \ge |k|$, which in turn follows from $k+m+n=0$. Indeed, if $|n|> 2|k|$, then we would have $|m|= |n+k| \ge |n|- |k| > |k|$ contradicting with $|k|\ge |m|$. Then, changing the order of summation and using Cauchy's inequality,
  $$\aligned
  J_1 &\le \sum_{i\ge 0} \sum_{n\in S_i\cup S_{i+1}} |u_n| \sum_{k\in S_i} \chi(k,-k-n,n) \, |u_k|\, |k+n|\, |v_{-k-n}|\\
  &\le \sum_{i\ge 0} \sum_{n\in S_i\cup S_{i+1}} |u_n| \bigg[\sum_{k\in S_i} |u_k|^2 |k+n|^2 |v_{-k-n}|^2 |k|^\alpha \bigg]^{\frac12} \bigg[\sum_{k\in S_i} \chi(k,-k-n,n) |k|^{-\alpha} \bigg]^{\frac12} \\
  &\le \sqrt{C_0} \sum_{i\ge 0} 2^{i\beta /2} \sum_{n\in S_i\cup S_{i+1}} |u_n| \bigg[\sum_{k\in S_i} |u_k|^2 |k+n|^2 |v_{-k-n}|^2 |k|^\alpha \bigg]^{\frac12},
  \endaligned $$
where the last step follows from \eqref{eq:dyadic-estimate}. Again by Cauchy's inequality,
  $$\aligned
  J_1 &\le \sqrt{C_0} \sum_{i\ge 0} 2^{i\beta /2} \bigg[\sum_{n\in S_i\cup S_{i+1}} |u_n|^2 \bigg]^{\frac12} \bigg[\sum_{n\in S_i\cup S_{i+1}} \sum_{k\in S_i} |u_k|^2 |k+n|^2 |v_{-k-n}|^2 |k|^\alpha \bigg]^{\frac12} \\
  &\le \sqrt{C_0} \sum_{i\ge 0} 2^{i\beta /2} \bigg[\sum_{n\in S_i\cup S_{i+1}} |u_n|^2 \bigg]^{\frac12} \bigg[ \sum_{k\in S_i} |u_k|^2 |k|^\alpha \sum_{m} |m|^2 |v_{m}|^2 \bigg]^{\frac12} \\
  &\le \sqrt{C_0} \|v\|_{H^1} \sum_{i\ge 0} \bigg[\sum_{n\in S_i\cup S_{i+1}} |n|^{\beta} |u_n|^2 \bigg]^{\frac12} \bigg[ \sum_{k\in S_i} |u_k|^2 |k|^\alpha \bigg]^{\frac12} ,
  \endaligned $$
where we have used $2^i\le |n|$ for every $n\in S_i\cup S_{i+1}$. Now Cauchy's inequality implies
  $$\aligned
  J_1 &\le \sqrt{C_0} \|v\|_{H^1} \bigg[ \sum_{i\ge 0}\sum_{n\in S_i\cup S_{i+1}} |n|^{\beta} |u_n|^2 \bigg]^{\frac12} \bigg[ \sum_{i\ge 0} \sum_{k\in S_i} |u_k|^2 |k|^\alpha \bigg]^{\frac12} \\
  &\le \sqrt{2C_0} \|v\|_{H^1} \|u\|_{H^{\alpha/2}} \|u\|_{H^{\beta/2}}.
  \endaligned $$

\emph{Step 2: estimate of $J_2$}. Similarly as above,
  $$\aligned
  J_2 &= \sum_n \sum_{m, |n|\ge |m|\ge |k|} \tilde\chi(k,m,n)\, |u_n|\, |u_{k}|\, |m|\, |v_m| \\
  &= \sum_{i\ge 0} \sum_{n\in S_i} \sum_{m, |n|\ge |m|\ge |k|} \tilde\chi(k,m,n)\, |u_n|\, |u_{k}|\, |m|\, |v_m| \\
  &\le \sum_{i\ge 0} \sum_{n\in S_i} \sum_{m\in S_{i-1}\cup S_{i}} \chi(-m-n,m,n)\, |u_n|\, |u_{-m-n}|\, |m|\, |v_m| ,
  \endaligned $$
where the last inequality is due to $2|m| \ge |n| \ge |m|$. Hence,
  $$\aligned
  J_2 &\le \sum_{i\ge 0} \sum_{m\in S_{i-1}\cup S_{i}} |m|\, |v_m| \sum_{n\in S_i} \chi(-m-n,m,n)\, |u_n|\, |u_{-m-n}| \\
  &\le \sum_{i\ge 0} \sum_{m\in S_{i-1}\cup S_{i}} |m|\, |v_m| \bigg[ \sum_{n\in S_i} |u_n|^2 |u_{-m-n}|^2 |n|^\alpha \bigg]^{\frac12} \bigg[ \sum_{n\in S_i} \chi(-m-n,m,n)\, |n|^{-\alpha} \bigg]^{\frac12} \\
  &\le \sqrt{C_0} \sum_{i\ge 0} 2^{i\beta/2} \sum_{m\in S_{i-1}\cup S_{i}} |m|\, |v_m| \bigg[ \sum_{n\in S_i} |u_n|^2 |u_{-m-n}|^2 |n|^\alpha \bigg]^{\frac12}
  \endaligned $$
where we have used \eqref{eq:dyadic-estimate}. By Cauchy's inequality,
  $$\aligned
  J_2 &\le \sqrt{C_0} \sum_{i\ge 0} 2^{i\beta/2} \bigg[ \sum_{m\in S_{i-1}\cup S_{i}} |m|^2 |v_m|^2 \bigg]^{\frac12} \bigg[ \sum_{m\in S_{i-1}\cup S_{i}} \sum_{n\in S_i} |u_n|^2 |u_{-m-n}|^2 |n|^\alpha \bigg]^{\frac12} \\
  &\le \sqrt{C_0} \sum_{i\ge 0} 2^{i\beta/2} \bigg[ \sum_{m\in S_{i-1}\cup S_{i}} |m|^2 |v_m|^2 \bigg]^{\frac12} \bigg[ \sum_{n\in S_i} |u_n|^2 |n|^\alpha \sum_{k} |u_{k}|^2 \bigg]^{\frac12} \\
  &\le \sqrt{C_0} \|u\|_{L^2} \sum_{i\ge 0} \bigg[ \sum_{m\in S_{i-1}\cup S_{i}} |m|^2 |v_m|^2 \bigg]^{\frac12} \bigg[ \sum_{n\in S_i} |u_n|^2 |n|^{\alpha+\beta} \bigg]^{\frac12}
  \endaligned $$
since $2^i \le |n|$ for all $n\in S_i$. Applying again Cauchy's inequality yields
  $$\aligned
  J_2 &\le \sqrt{C_0} \|u\|_{L^2} \bigg[ \sum_{i\ge 0} \sum_{m\in S_{i-1}\cup S_{i}} |m|^2 |v_m|^2 \bigg]^{\frac12} \bigg[ \sum_{i\ge 0} \sum_{n\in S_i} |u_n|^2 |n|^{\alpha+\beta} \bigg]^{\frac12} \\
  &\le \sqrt{2C_0} \|u\|_{L^2} \|v\|_{H^1} \|u\|_{H^{(\alpha+\beta)/2}} .
  \endaligned $$

Finally, the estimate of $J_5$ is similar to that of $J_2$. So we finish the proof of Lemma \ref{lem-restricted-convol} by summarizing the above estimates.
\end{proof}

\bigskip

\noindent \textbf{Acknowledgements:} The author would like to thank Professor Franco Flandoli for helpful discussions. He is grateful to the financial supports of the National Key R\&D Program of China (No. 2020YFA0712700), the National Natural Science Foundation of China (Nos. 11931004, 12090014) and the Youth Innovation Promotion Association, CAS (Y2021002).

\end{document}